\documentclass[11pt, oneside]{article}
\usepackage{amssymb}
\usepackage{amsthm}
\usepackage{amsmath}
\usepackage{xcolor}
\usepackage{graphicx}
\usepackage{enumitem}

\usepackage{geometry}

\newtheorem{theorem}{Theorem}

\newtheorem{lemma}{Lemma}
\newtheorem{conjecture}{Conjecture}

\title{List-Edge Coloring with bounded Maximum Average Degree}
\author{Joshua Harrelson\footnote{joshua.harrelson@mga.edu, ORCID: 0000-0002-9749-374X }
}

\begin{document}

\maketitle

\begin{abstract}
For a graph $G$, we show that if $mad(G)<m$, then $\chi'_\ell(G)\leq \Delta+1$ where $m$ depends upon $\Delta$ and $\chi'_\ell(G)$ is the list-chromatic index of $G$. When $\Delta\leq 20$ the value of $m$ is close to $\frac{1}{2}\Delta$, but as $\Delta$ increases $m$ becomes asymptotic to about $\frac{1}{4}\Delta+5$.
\end{abstract}

\section{Introduction}
We consider only simple graphs in this paper which is an extension of \cite{HR}, and we will introduce similar definitions to this citation for convenience. The \textit{average degree} of a graph $G$ is $ad(G)=\frac{\sum deg(v)}{v(G)}$, and the \textit{maximum average degree} of a graph $G$ is $mad(G)=max\{ad(H): H\subseteq G\}$. 

We say an \emph{edge-coloring} of $G$ is a function which maps one color to every edge of $G$ such that the same color is not assigned to adjacent edges. A $k$-edge-coloring of $G$ is an edge-coloring of $G$ which maps a total of $k$ colors to $E(G)$ and the chromatic index $\chi'(G)$ is the minimum $k$ such that $G$ is $k$-edge-colorable. If $G$ is a graph with maximum degree $\Delta(G)=\Delta$, then Vizing's Theorem \cite{Viz} gives us $\chi'(G)\leq \Delta+1$. 

An \emph{edge-list-assignment} of $G$ is a function which maps a set of colors to every edge in $G$. If $L$ is an edge-list-assignment of $G$, then we say the set of colors mapped to $e \in E(G)$ is the list, $L(e)$. We say that $G$ is \emph{L-colorable} if $G$ can be properly edge-colored such that every edge $e$ receives a color from $L(e)$. We say that $G$ is \emph{$k$-list-edge-colorable} if $G$ is $L$-colorable for all $L$ such that $|L(e)|\geq k$ for all $e \in E(G)$. The list-chromatic index, $\chi'_\ell(G)$, is the minimum $k$ such that $G$ is $k$-list-edge-colorable. This is to say that $\chi'_\ell(G)$ seeks a list-edge-coloring for all list-assignments with list-size at least $k$. We note that \emph{$k$-list-edge-colorable} is referred to as \emph{$k$-edge-choosable} in other papers.  

We know that $\Delta\leq \chi'(G) \leq \chi'_\ell(G)$ for all graphs. The List-Edge Coloring Conjecture (LECC) proposes that $\chi'_\ell(G)=\chi'(G)$, but this has only been verified for a few special families of graphs, notably bipartite graphs due to Galvin \cite{Gal}. In this paper, we will focus on a relaxation of the LECC proposed by Vizing.

\begin{conjecture}[Vizing \cite{Vz}]
\label{vizcon}
If $G$ is a graph, then $\chi'_\ell(G) \leq \Delta+1$.
\end{conjecture}

Vizing verified this conjecture when $\Delta=3$ in 1976 \cite{Vz} and this was also independently verifying by Erd\H{o}s, Rubin, and Taylor in 1979 \cite{ERT}.
Juvan, Mohar, \v{S}krekovski verified the conjecture when $\Delta=4$ in 1998 \cite{JMS}.

Vizing's conjecture has generated interest for the family of planar graphs. Borodin \cite{B} verified Conjecture \ref{vizcon} for planar graphs with $\Delta\geq9$ in 1990 and this result was improved to planar graphs with $\Delta\geq8$ by Bonamy \cite{Bo} in 2015. Zhang \cite{Z} also verified Vizing's conjecture for planar graphs without triangles in 2004.

For a graph $G$, we let $V_x$ and $V_{[x,y]}$ be vertex sets such that $V_x=\{v\in V(G)\ | \ deg(v)=x\}$ and $V_{[x,y]}=\{v\in V(G)\ | \ x\leq deg(v) \leq y\}$. We say that a graph $G$ is $\emph{k-list-edge-critical}$ \ if $\chi'_\ell(G)>k$, and $\chi'_\ell(G-e)\leq k$ for all $e\in E(G)$. In 2010, Cohen and Havet \cite{CH} provided an alternate proof of Borodin's theorem which reduced the argument to about a single page. Their paper used the minimality of list-edge-critical graphs and a clever discharging argument. We state one of their lemmas below.

\begin{lemma}[Cohen \& Havet \cite{CH}]
\label{edgeweight}
If $G$ is $(\Delta+1)$-list-edge-critical, then $deg(u)+deg(v)\geq \Delta+3$.
\end{lemma}

Lemma \ref{edgeweight}, together with Borodin, Kostochka, Woodall's generalization \cite{BKW} of Galvin's Theorem, were used to prove the following lemma which is a simplified version of Lemma 9 in \cite{HMP}. 

\begin{lemma}[H., McDonald, Puleo \cite{HMP}]
\label{hmp} Let $a_0,a,b_0 \in \mathcal{N}$ such that $a_0>2$, $b_0>a$, and $a+b_0= \Delta+3$. If $G$ is $(\Delta+1)$-list-edge-critical, then 

$$2\sum_{i=a_0}^{a} |V_i| < \sum_{j=b_0}^{\Delta} (a+j-\Delta-2)|V_j|.$$
\end{lemma}

\section{Main Result}

The authors of \cite{HR} applied Lemma \ref{hmp} to a few values of $\Delta(G)$ to show the following.

\begin{theorem}[H., Reavis \cite{HR}]
If $G$ has $\Delta(G)=\Delta\leq9$ and $mad(G)<\frac{\Delta+3}{2}$, then $\chi'_\ell(G)\leq\Delta+1$.
\end{theorem}

By optimizing the use of Lemma \ref{hmp}, we can prove the following theorem which extends the previous result to improved values of $m$ and any value of $\Delta$ provided $m$ is ``large enough".

\begin{theorem}
\label{mad}
Let $G$ be a graph with $\Delta(G)=\Delta$ and $c=\lfloor \frac{1}{2}\Delta+1\rfloor$. If $mad(G)<m$, then $\chi'_\ell(G)\leq \Delta+1$ where
\begin{equation*}
m=
    \begin{cases}
        c+1 & \text{if } \Delta =5,7\\
        14/3 & \text{if } \Delta=6\\
        \frac{4\Delta+c^2+c-6}{2c}  & \text{if } \Delta\geq8\\
    \end{cases}
\end{equation*} 
\end{theorem}
\begin{proof}
We note this theorem is vacuously true when $\Delta\leq 4$ as discussed in the Introduction so we may assume $\Delta\geq 5$. Assume for contradiction that $G$ is $(\Delta+1)$-list-edge-critical and $mad(G)<m$. By Lemma \ref{edgeweight}, this means $\delta(G)\geq 3$. We define the following initial charges; let $\alpha(v)=deg(v)$ for all $v \in V(G)$ and let $\alpha(P)=0$ for an artificial, global pot $P$. Let $\alpha(G)$ denote the sum of all initial charges. We know $ad(G)=\frac{\sum deg(v)}{v(G)}$, rather $\alpha(G)=ad(G)\cdot v(G)< m\cdot v(G)$. We will apply a discharging step and denote $\alpha'(x)$ as the final charge for $x$ after this step. To get a contradiction, we will prove $\alpha'(G)> m \cdot v(G)$ by showing $\alpha'(P)>0$ and $\alpha'(v)\geq m$ for all $v \in V(G)$. 

We will discharge according to sets of ordered triples and their corresponding inequalities from Lemma \ref{hmp} that are displayed in Table \ref{ineq}. Let $I'_j=\frac{1}{2}I_j$ for $1\leq j\leq c-3$ and let $I'_{c-2}=\frac{1}{2}(m-c)I_{c-2}$. Let $x_j$ be the sum of coefficients of $|V_j|$ occurring in the system of inequalities $I'=\{I'_1,I'_2,...I'_{c-2}\}$. 

\noindent We apply the discharging step by the following rules:
\begin{enumerate}
    \item If $deg(v)=j>c$, then $v$ will give $x_j$ charge to $P$.
    \item If $deg(v)=j\leq c$, then $v$ will take $x_j$ charge from $P$. 
\end{enumerate}

 \renewcommand{\arraystretch}{1.5}
 \begin{table}[h]
    \centering
    \begin{tabular}{|c|c|c|}
     
        \hline
        Triple: $(a_0,a,b_0)$ & Inequality & Label\\
        \hline
         $(3,3,\Delta)$& $2|V_3| < |V_\Delta|$ & $I_1$\\[.2cm]
         $(3,4,\Delta-1)$& $2|V_3|+2|V_4|<|V_{\Delta-1}|+2|V_\Delta|$ & $I_2$\\[.2cm]
         $(3,5,\Delta-2)$& $2|V_3|+2|V_4|+2|V_5|<|V_{\Delta-2}|+2|V_{\Delta-1}|+3|V_\Delta|$ & $I_3$\\[.2cm]
         \vdots &  \vdots & \vdots\\[.2cm]
         $(3,c,\Delta)$ & $2|V_3|+2|V_4|+\cdots+2|V_c|<|V_{\Delta-c+3}|+2|V_{\Delta-c+4}|+\cdots+(c-2)|V_\Delta|$ & $I_{c-2}$\\[.2cm]
          \hline
        
    \end{tabular}
     \caption{Triples and resulting inequalities from Lemma \ref{hmp}.}
    \label{ineq}
 \end{table}

Discharging rule 1 shows that $P$ receives $x_j$ for all $j\geq \Delta-c+3$ since $\Delta-c+3>c$. This means that more charge is given to $P$ than taken from $P$ since every equation in $I'$ is strict, so $\alpha'(P)>0$. 

If $v\in V_{[c+1,\ \Delta-c+2]}$, then $\alpha'(v)=deg(v)$ as these vertices do not appear in Table \ref{ineq}. If $5 \leq \Delta \leq 7$, then $\alpha'(v)\geq c+1\geq m$. If $\Delta\geq 8$ is even, then $\Delta=2c-2$, $c\geq 5$, and $m=\frac{c^2+9c-14}{2c}<c+\frac{2}{3}<\alpha'(v)$. If $\Delta\geq8$ is odd, then $\Delta=2c-1$, $c\geq 5$, and $m=\frac{c^2+9c-10}{2c}\leq c+1\leq \alpha'(v)$.

 \renewcommand{\arraystretch}{1.5}
      \begin{table}[]
          \centering

      \begin{tabular}{|c|c|}
      \hline
      $V_j$ & $a'(V_j)$\\[.2cm]
      \hline
      $V_3$ & $3+(m-c)+\sum_{i=1}^{c-3}1$\\[.2cm]
      $V_4$ & $4+(m-c)+\sum_{i=1}^{c-4}1$\\[.2cm]
      \vdots & \vdots\\
      $V_c$ & $c+(m-c)$\\[.2cm]
      $V_{\Delta-c+3}$ & $(\Delta-c+3)-\frac{1}{2}(m-c)$\\[.2cm]
      $V_{\Delta-c+4}$ & $(\Delta-c+4)-\frac{2}{2}(m-c)-\frac{1}{2}\sum_{i=1}^{1}i$\\[.2cm]
      $V_{\Delta-c+5}$ &$(\Delta-c+5)-\frac{3}{2}(m-c)-\frac{1}{2}\sum_{i=1}^{2}i$\\[.2cm]
      \vdots & \vdots \\
      $V_\Delta$ & $\Delta-\frac{c-2}{2}(m-c)-\frac{1}{2}\sum_{i=1}^{c-3}i$\\
      \hline
      \end{tabular}
      \caption{The final charges for $v\in V_{[3,c]}\cup V_{[\Delta-c+3,\ \Delta]}$}
      \label{fc}
      \end{table}
To complete the proof we will show that $\alpha'(v)\geq m$ for $v \in V_{[3,c]}\cup V_{[\Delta-c+3,\Delta]}$. The final charges of these vertices are given in Table \ref{fc} according to the previously defined discharging rules. 

If $v\in V_c$\ , then $\alpha'(v)=c+(m-c)=m$. If $u \in V_{c-i}$ for $1\leq i \leq c-3$, then $\alpha'(u)=c-i+m-c+\sum_{j=1}^{c-(c-i)} 1=m$. We will now proceed through the values of $\Delta$ and show for any $v \in V_{[\Delta-c+3,\Delta]}$ that $\alpha'(v)\geq m$.

Let $\Delta=5$, then $c=3$, $m=4$, and $V_{[\Delta-c+3,\Delta]}=V_5$. If $v\in V_5$, then we see from Table \ref{fc} that $\alpha'(v)=4.5>m$.

Let $\Delta=6$, then $c=4$, $m=\frac{14}{3}$, and $V_{[\Delta-c+3,\Delta]}=V_{[5,6]}$. If $v \in V_5$ and $u\in V_6$, then $\alpha'(v)=\frac{14}{3}=m$ and $\alpha'(u)=\frac{29}{6}>m$.

Let $\Delta=7$, then $c=4$, $m=5$, and $V_{[\Delta-c+3,\Delta]}=V_{[6,7]}$. If $v\in V_6$ and $u\in V_7$, then $\alpha'(v)=\alpha'(u)=5.5>m$.

Let $\Delta\geq 8$, then $m=\frac{4\Delta+c^2+c-6}{2c}$. Let $v\in V_{\Delta-c+3}$ and recall that $\alpha'(v)= \Delta-c+3-\frac{1}{2}(m-c)=\Delta-\frac{1}{2}c+3-\frac{1}{2}m$. We note that if $\frac{2}{3}\Delta-\frac{1}{3}c+2>m$, then $\alpha'(v)>m$. It follows that

\begin{equation*}
\frac{2}{3}\Delta-\frac{1}{3}c+2=
    \begin{cases}
        c+\frac{2}{3} & \text{if } \Delta \text{ is even}\\[.2cm]
        c+\frac{4}{3} & \text{if } \Delta \text{ is odd}\\
    \end{cases}
\end{equation*} 

We previously showed that if $\Delta \geq 8$ is even, then $m<c+\frac{2}{3}$. We also showed that if $\Delta\geq8$ is odd, then $m\leq c+1<c+\frac{4}{3}$. So it must be that $\alpha'(v)>m$.

Let $u\in V_\Delta$. We have
\begin{align*}
    \alpha'(u)&=\Delta-\frac{c-2}{2}(m-c)-\frac{1}{2}\sum_{i=1}^{c-3}i\\[.2cm]
    &=\Delta+\frac{c^2-2c}{2}-\frac{1}{2}\cdot\frac{(c-3)(c-2)}{2}-\frac{c-2}{2}m\\[.2cm]
    &=\frac{4\Delta+c^2+c-6}{4}-\frac{c-2}{2}m\\[.2cm]
    &=\frac{c}{2}m-\frac{c-2}{2}m\\[.2cm]
    &=m\\
\end{align*}

If $w \in V_{\Delta-c+k}$ for $3\leq k \leq c$, then
\begin{align*}
    \alpha'(w) &= \Delta-c+k-\frac{k-2}{2}(m-c)-\frac{1}{2}\sum_{j=0}^{k-3}j\\[.2cm]
    &= \Delta-c+k-(\frac{1}{2}mk-\frac{1}{2}ck-m+c)-\frac{1}{2}\cdot\frac{(k-3)(k-2)}{2}\\[.2cm]
    &= \Delta-c+k-\frac{1}{2}mk+\frac{1}{2}ck+m-c-\frac{1}{4}k^2+\frac{5}{4}k-\frac{3}{2}\\[.2cm]
    &= -\frac{1}{4}k^2+(\frac{9}{4}-\frac{1}{2}m+\frac{1}{2}c)k+(\Delta-2c+m-\frac{3}{2})\\
\end{align*}

We see that $\alpha'(w)$ behaves as a concave quadratic function with respect to $k$. We also know $\alpha'(v)>m$ and $\alpha'(u)=m$, meaning $m$ is the minimum final charge for all vertices in $V_{[\Delta-c+3,\Delta]}$. 
\end{proof}

\section{Conclusion}

Theorem \ref{mad} can verify Conjecture \ref{vizcon} for many sparse graphs. For instance, Theorem \ref{mad} says that if a graph $G$ has $\Delta(G)\geq9$ and $mad(G)<6$, then the conjecture is true. It is known that any planar graph $G$ has $mad(G)<6$ so Borodin's theorem is verified. Theorem \ref{mad} also says that if a graph $G$ has $mad(G)<4$, then the conjecture is true. It is known that any planar graph $G$ without a triangle has $mad(G)<4$ so Zhang's theorem is also verified. 

When $\Delta\leq 20$, the value $m=\frac{4\Delta+c^2+c-6}{2c}$ from Theorem \ref{mad} is close to $\frac{1}{2}\Delta$. As Delta continues to increase, $m$ becomes asymptotic to $\frac{1}{4}\Delta+5$ when $\Delta$ is even and $m$ becomes asymptotic to $\frac{1}{4}\Delta+\frac{21}{4}$ when $\Delta$ is odd.


\newpage
\begin{thebibliography}{99}

\bibitem{Bo}
M. Bonamy, Planar graphs with $\Delta \geq 8$ are $(\Delta+1)$-edge-choosable, \emph{Seventh Euro. Conference in Comb., Graph Theory and App., CRM series, Edizioni della Normale} {\bf 16} (2013). https://doi.org/10.1137/130927449
\bibitem{BKW}
O.V. Borodin, A. V. Kostochka, and D. R. Woodall, List edge and list total colorings of multigraphs, \emph{J. Combin. Theory Ser. B} \textbf{71} (1997), 184-204. https://doi.org/10.1006/jctb.1997.1780
\bibitem{B}
O.V. Borodin, A generalization of Kotzig's theorem on prescribed edge coloring of planar graphs, \emph{Mat. Zametki} \textbf{48} (1990), 1186-1190. https://doi.org/10.1007/BF01240258
\bibitem{CH}
N. Cohen, F. Havet, Planar graphs with maximum degree $\Delta \leq 9$ are $(\Delta+1)$-edge-choosable-a short proof, \emph{Discrete Math.} \textbf{310}  (2010), 3049-3051. https://doi.org/10.1016/j.disc.2010.07.004
\bibitem{ERT}
P. Erd\~{o}s, A. Rubin, H. Taylor, Choosability in graphs. \emph{Congr. Numer.} \textbf{26} (1979), 125-157.
\bibitem{JMS}
M. Juvan, B. Mohar, R. \v{S}krekovski, List total colorings of graphs, \emph{Combin. Probab. Comput.} \textbf{7} (1998), 181-188. https://doi.org/10.1017/S0963548397003210
\bibitem{Vz}
V. Vizing, Colouring the vertices of a graph with prescribed colours, \emph{Diskret. Analiz} \textbf{29} (1976), 3-10. (In Russian)
\bibitem{Viz}
 V. G. Vizing, On an estimate of the chromatic class of a p-graph, \emph{Diskret. Analiz} \textbf{3} (1964), 25–30. 
\bibitem{HMP}
J. Harrelson, J. McDonald, G. Puleo, List-edge-colouring planar graphs with precoloured edges, \emph{European J. of Combin.} {\bf 75} (2019), 55-65. https://doi.org/10.1016/j.ejc.2018.07.003

\bibitem{Gal}
F. Galvin, The List Chromatic Index of a Bipartite Multigraph, \textit{Journal of Combinatorial Theory, Ser. B} \textbf{63} (1995), 153-158. https://doi.org/10.1006/jctb.1995.1011
\bibitem{Z}
Zhang, L., Wu, B.: Edge choosability of planar graphs without small cycles. Discrete Math. 283, 289-293 (2004). https://doi.org/10.1016/j.disc.2004.01.001


\bibitem{HR}
J. Harrelson, H. Reavis, Maximum Average Degree of List-Edge-Critical Graphs and Vizing's Conjecture. \emph{Elec. J. of Graph Theory and App.} \textbf{10} (2022).
\end{thebibliography}
\end{document}